\documentclass[11pt]{amsart}

\usepackage{amsmath,amssymb,amsfonts,graphics}

\newcommand{\N}{\mathbb{N}}
\newcommand{\Z}{\mathbb{Z}}
\newcommand{\R}{\mathbb{R}}
\newcommand{\C}{\mathbb{C}}
\newcommand{\om}{\omega}
\newcommand{\sg} {\sigma}
\newcommand{\im}{\text{im}}
\newcommand{\f}{{\bf f}}
\newcommand{\g}{{\bf g}}
\newcommand{\h}{{\bf h}}
\newcommand{\spn}{\mathrm{span}}
\newcommand{\oo}{\tilde{\omega}}
\newcommand{\os}{\tilde{\sigma}}
\newcommand{\D}{\tilde{D}}
\newcommand{\T}{\widehat{T}}
\newcommand{\kf}{\mathfrak{f}}
\newcommand{\kg}{\mathfrak{g}}
\newcommand{\X}{\mathfrak{X}}
\newcommand{\To}{\longrightarrow}
\newcommand{\dss}{\displaystyle}
\newcommand{\from}{\colon}
\newcommand{\supp}{\text{supp}}

\newcommand{\la}{\langle}
\newcommand{\ra}{\rangle}

\newtheorem{thm}{Theorem}[section]
\newtheorem{cor}[thm]{Corollary}
\newtheorem{lem}[thm]{Lemma}
\newtheorem{prop}[thm]{Proposition}
\theoremstyle{definition}
\newtheorem{defn}[thm]{Definition}
\theoremstyle{remark}

\newtheorem{rem}[thm]{Remark}
\numberwithin{equation}{section}

\title [weak amenability and 2-weak amenability of Beurling algebras]
{weak amenability and 2-weak amenability of Beurling algebras}

\author[Ebrahim Samei]{Ebrahim Samei}
\address{Ebrahim Samei, Department of Pure Mathematics, Faculty of Mathematics, University of Waterloo, Waterloo, ON, Canada, N2L 3G1}

\subjclass{Primary 43A20, 47B47.} \keywords{locally compact groups, group algebras, weight function, Beurling algebras, derivations, weakly amenable, 2-weakly amenable}

\thanks{}

\email{esamei@math.uwaterloo.ca}

\begin{document}

\maketitle

\begin{abstract}
Let $L^1_\om(G)$ be a Beurling algebra on a locally compact abelian group $G$. We look for general conditions on the weight which allows the vanishing of continuous derivations of $L^1_\om(G)$. This leads us to introducing vector-valued Beurling algebras and considering the translation of operators on them. This is then used to connect the augmentation ideal to the behavior of derivation space. We apply these results to give examples of various classes of Beurling algebras which are weakly amenable, 2-weakly amenable or fail to be even 2-weakly amenable.
\end{abstract}

Let $A$ be a Banach algebra, let $n\geq 0$ be an integer, and let $A^{(n)}$ be the $n$-th dual module of $A$ when $n>0$, and be $A$ itself when $n=0$. The algebra $A$ is said to be {\it weakly amenable} if bounded derivations $D\from A \to A^*$ are inner, and it is said to be {\it n-weakly amenable} if bounded derivations $D\from A \to A^{(n)}$ are inner. The algebra $A$ is {\it permanently weakly amenable} if it is $n$-weakly amenable for all $n\geq1$.

The concept of weak amenability was first introduced by Bade, Curtis and Dales \cite{BCD} for commutative Banach algebras, and was extended to the non-commutative case by B. E. Johnson \cite{J1} and it has been the object of many studies since (see for example, \cite{FM}, \cite{G3}, and \cite{YZ} and references therein). Dales, Ghahramani and Gr{\o}nb{\ae}k initiated the study of $n$-weakly amenable Banach algebras in \cite{DGG}, where they revealed many important properties of these algebras and presented some examples of them. For instance, they showed that
$C^*$-algebras are permanently weakly amenable; the fact that was known for weakly amenable commutative Banach algebras \cite[Theorem 1.5]{BCD}. They also showed that group algebras are $2n+1$-weakly amenable for all $n>0$
(for more example see \cite{FM}, \cite{J3}, and \cite{YZ}).

Let $L^1_\om(G)$ be a Beurling algebra on a locally compact abelian group $G$. One can pose the question of whether $L^1_\om(G)$ is $n$-weakly amenable; in our case it means that each derivation from $L^1_\om(G)$ into $L^1_\om(G)^{(n)}$ is zero. The case of weak amenability has been studied in \cite{BCD} and \cite{G1}. One major result states that $l^1_\om(\Z)$ is weakly amenable if and only if $\inf_n \frac{\om(n)\om(-n)}{n}=0$ \cite{G1}. From this, it can be easily deduced that $l^1_\om(G)$ is weakly amenable if $\inf_n \frac{\om(nt)\om(-nt)}{n}=0$ for all $t\in G$. 

Now it is natural to ask on what condition on the weight $\om$, $L^1_\om(G)$ is 2-weakly amenable. In their recent memoir \cite{DL}, Dales and Lau have addressed this question. They show that if $\om\geq 1$ is almost invariant and satisfies $\inf_n \frac{\om(nt)}{n}=0$ for all $t\in G$, then $L^1_\om(G)$ is 2-weakly amenable \cite[Theorem 13.8]{DL}, and conjecture that $L^1_\om(G)$ is 2-weakly amenable if one only assumes that $\inf_n \frac{\om(nt)}{n}=0$ ($t\in G$). Ghahramani and Zabandan proved the conjecture with an additional condition which is weaker than being almost invariant \cite{GZ}.

The central goal of this paper is to study systematically, for $G$ abelian, the behavior of the derivation space of $L^1_\om(G)$ into the dual of an arbitrary symmetric Banach $L^1_\om(G)$-module $X$ and to see when it vanishes.
From the fundumental work of Gr{\o}nb{\ae}k \cite{G1}, this question can be reduced to studying the kernel of the multiplication map from $L^1_\om(G)\widehat{\otimes} X$ into $X$. In the case of group algebras i.e. when $\om=1$, it is well-known that this can be done by transfering the properties of the augmentation ideal of $L^1(G)$ into the multiplication ideal of $L^1(G)\widehat{\otimes} L^1(G) \cong L^1(G \times G)$ through the isometric isomorphism
$$\gamma \from L^1(G)\widehat{\otimes} L^1(G) \to L^1(G)\widehat{\otimes} L^1(G) \ \ , \ \gamma(m)(s,t)=m(s,st) \eqno{(*)},$$
and then deduce it for any module (\cite[Theorem 1.8]{S} and \cite[Theorem 2.9.65]{D}). However, in general, this idea can not be applied in its present form to the Beurling algebra $L^1_\om(G)$ because the map $\gamma$ may not be well-defined if we replace $L^1(G)$ with $L^1_\om(G)$ in $(*)$. Our approach is to consider a translation map similar to $\gamma$ but on
$L^1_\om(G)\widehat{\otimes} X$. This will allow us to look directly at the kernel of the multiplication map on $L^1_\om(G)\widehat{\otimes} X$ instead of relying on $L^1_\om(G)\widehat{\otimes} L^1_\om(G)$. However, in order to  do this, we need to consider vector-valued integration for Beurling algebras.

In Section \ref{S: Vec-BA}, we introduce the concept of a vector-valued Beurling algebra $L^1_\om(G,A)$ and the module $L^1_\om(G,X)$, where $A$ is an arbitrary Banach algebra and $X$ is a Banach left $A$-module. We show that  $L^1_\om(G)\widehat{\otimes} X\cong L^1_\om(G,X)$ can be regarded isometrically as a module over $L^1_\om(G,A)$.
When $\om=1$, this concept has been thoroughly developed in \cite[Chapter VIII]{FD} for cross-sectional algebras to construct examples of cross-products of C$^*$-algebras and C$^*$-algebra bundles.

For the rest of the paper, we restrict ourselves to the case when $G$ is abelian. In Sections \ref{S:Trans-op} and \ref{S:der space-aug ideal}, we use the idea in Section \ref{S: Vec-BA} for the case where $A$ is the Beurling measure algebra on a weight $\sg$ and $X$ is symmetric, to define a translation map such as $(*)$ from $L^1_{\om\os}(G)\widehat{\otimes} X$ into $L^1_\om(G)\widehat{\otimes} X$,
where $\os(t)=\sg(-t)$. This, in most of the desirable cases, connects the augmentation ideal of $L^1_{\om\os}(G)$ to the kernel of the multiplication map on $L^1_\om(G)\widehat{\otimes} X$ (Theorem \ref{T: Kernel-Agumentation}). Eventually we demonstrate that if the augmentation ideal of $L^1_{\om\os}(G)$ is essential, or equivalently, if there is no non-zero, continuous point derivation on the augmentation character, then the derivation space from  $L^1_\om(G)$ into $X$ vanishes
(Theorem \ref{T:Beurling alg-der}).

The reminder of this paper is devoted to investigating the weak amenability and 2-weak amenability of $L^1_\om(G)$ by applying the preceding results. 

In Section \ref{S:weak amen}, we shall show that $L^1_\om(G)$ is weakly amenable if $\inf \{ \Omega (nt)/n \mid n\in \N \}=0$ for all $t\in G$, where $\Omega(t):=\omega (t)\omega (-t)$. This follows from the observation that
the above assumption implies that there is no non-zero, continuous point derivation on the augmentation character
of $L^1_\Omega(G)$. This result extends the result of Gr{\o}nb{\ae}k and provides an alternative proof of it.

For a weight $\om\geq 1$, let $\om_1(s)=\dss\limsup_{t\to\infty} \frac{\om(t+s)}{\om(t)}$. In \cite{GZ}, it is shown that
$L^1_\om(G)$ is 2-weakly amenable if $\inf \{ \omega (nt)/n \mid n\in \N \}=0$ and $\om_1$ is bounded. In Section
\ref{S:2-weak amen}, we first show that there a weight $\sg_\om$ on $G$ which is closely related to $\om_1$. We then show that the result in \cite{GZ} mentioned above is a particular case of the fact that 2-weak amenablility of
$L^1_\om(G)$ follows if there is no non-zero, continuous point derivation on the augmentation character
of $L^1_{\omega\tilde{\sg}_\om}(G)$. Moreover, when $\sg_\om$ is bounded, there is a precise correspondence between the essentiality of the augmentation ideal of $L^1_\om(G)$ and 2-weak amenability of $L^1_\om(G)$ (Theorem \ref{T:2-weak amenable}). This fact allows us to classify various classes of weights for which their corresponding Beurling algebras are 2-weakly amenable or fail to be 2-weakly amenable. These weights, which are defined on compactely generated abelian groups, include polynomial weights, exponential weights, and certain weights satisfying condition (S) (Section  \ref{S:compactly generated groups}). For instance, for non-compact groups, we show that a Beurling algebra of a polynomial weight of degree $\alpha$ is 2-weakly amenable if and only if $0\leq \alpha <1$, whereas a Beurling algebra of an exponential weight of degree $\alpha$ is never 2-weakly amenable if $0< \alpha <1$. We extend the later result to a much larger class of symmetric weights for which the growth is exponential. However, we give examples of families of
non-symmetric weights with sharp exponential growth, for which the Beurling algebras are 2-weakly amenable.

\section{Preliminaries}\label{P}

Let $A$ be a Banach algebra, and let $X$ be a Banach $A$-bimodule. An operator $D\from A \to X$ is a
{\it derivation} if for all $a,b\in A$, $D(ab)=aD(b)+D(a)b$. For each
$x\in X$, the operator $ad_x\in B(A,X)$ defined by $ad_x(a)=ax-xa$
is a bounded derivation, called an {\it inner derivation}. Let $\mathcal{Z}^1(A,X)$ be the linear space of
all bounded derivations from $A$ into $X$. When $A$ is commutative, a Banach $A$-bimodule $X$ is {\it symmetric} if for all $a\in A$ and $x\in X$, $ax=xa$. In this case, we say simply that $X$ is a Banach
$A$-module.

Let $G$ be a locally compact group with a fixed left Haar measure
$\lambda$. The measure algebra $M(G)$ is the Banach space of
complex-valued, regular Borel measures on $G$. The space $M(G)$ is
identified  with the (dual) space of all continuous linear
functionals on the Banach space $C_0(G)$, with the duality
specified by setting
$$ \la \mu \ , \ f \ra=\int_G f(t)d\mu(t) \ \ \ (f\in C_0(G), \mu\in
M(G)).$$ The convolution multiplication $*$ on $M(G)$ defined by
setting
$$ \la \mu*\nu \ , \ f \ra=\int_G\int_G f(st)d\mu(s)d\nu(t) \ \ \ (f\in
C_0(G),\ \mu,\nu \in M(G)). $$ We write $\delta_s$ for the point
mass at $s\in G$; the element $\delta_e$ is the identity of
$M(G)$, and $l^1(G)$ is the closed subalgebra of $M(G)$ generated by the
point masses. Then $M(G)$ is a unital Banach algebra and $L^1(G)$,
the group algebra on $G$, is a closed ideal in $M(G)$
\cite[Theorem 3.3.36]{D}. Moreover, the dual of $L^1(G)$ can be identified with $L^\infty(G)$, the Banach space of Borel measuarable essentially bounded functions on $G$. We let $LUC(G)$ denote the closed subspace of $L^\infty(G)$ consisting of the (equivalence classes of) bounded left uniformly continuous functions on $G$.

Let $G$ be a locally compact group with identity $e$. A weight on $G$ is a continuous function
$\om \from G \to (0,\infty )$ such that
$$ \om (st)\leqslant \om (s)\om (t) \ \ \ (s,t \in G), \ \ \ \om (e)=1.$$

Let $X$ be a Banach space of measures or of equivalence classes of functions on a locally compact
group $G$, and let $\om \from G \to (0,\infty )$ be a continuous function. We define the Banach space
$$X(\om):=\{f \mid \om f \in X \},$$
where the norm of $X(\om)$ is defined so that the map $f \mapsto \om f$ from $X(\om)$ onto $X$
 is a linear isometry. In particular, we let $M_\om(G):=M(G)(\om)$, $L^1_\om(G):=L^1(G)(\om)$,
$l^1_\om(G):=l^1(G)(\om)$, $L^\infty_{1/\om}(G):=L^\infty(G)(1/\om)$, $LUC_{1/\om}(G):=LUC(G)(1/\om)$, and $C_{0,1/\om}(G):=C_0(G)(1/\om)$. When $\om$ is a weight, with the convolution multiplication of measures, $M_\om(G)$ becomes a Banach algebra, having $L^1_\om(G)$ as a closed two-sided ideal and $l^1_\om(G)$ as a closed subalgebra. Moreover,  $M_\om(G)=L_\om^1(G)=l^1_\om(G)$ if and
only if $G$ is discrete. Also $L^\infty_{1/\om}(G)$ is the dual of $L^1_\om(G)$, having $LUC_{1/\om}(G)$ and $C_{0,1/\om}(G)$ as Banach $L^1_\om(G)$-submodules. The algebras $L^1_\om(G)$ are the {\it Beurling algebras} on $G$. For more details see \cite[Chapter 7]{DL}.

\section{Vector-valued Beurling algebras}\label{S: Vec-BA}

Let $G$ be a locally compact group, let $\omega$ be a weight on $G$, and let $\om\lambda$ be the positive regular Borel measure on $G$
defined by $$\om\lambda(E)=\int_E \om(t)d\lambda(t), $$ where $E$ is an arbitrary $\lambda$-measurable set. It is well-known that
$\om\lambda$ is well-definded since $\om$ is continuous and positive.
Moreover, $E\subseteq G$ is $\om\lambda$-measurable if and only if $E$ is $\lambda$-measurable.

Our references for vector-valued integration theory is \cite{DF} and \cite{DU}. 
Let $(X, ||\cdot||_X)$ be a Banach space, and let $\mathfrak{L}^1_\om(G,X)$ be the set of all $\lambda$-measurable (or equivalently, $\om\lambda$-measurable) vector-valued functions $\kf \from G \to X$ such that $\int_G ||\kf(t)||\om(t)dt < \infty $ (see \cite[Appendix B.11]{DF} or \cite[Definition II.1.1]{DU} for the definition of vector-valued measurable functions). The functions in $\mathfrak{L}^1_\om(G,X)$ are called {\it Bochner $\om\lambda$-integrable} since for them the Bochner integral exists. It is clear that $\mathfrak{L}^1_\om(G,X)$ is a vector space with the standard addition and scalar multiplication.
Let $L^1_\om(G,X)$ be the equivalent classes of elements in
$\mathfrak{L}^1_\om(G,X)$ with respect to the semi-norm $\Vert \cdot \Vert=\int_G ||\cdot||_X\om(t)dt$, i.e.
$L^1_\om(G,X):=\mathfrak{L}^1_\om(G,X) / \sim$ where $\kf\sim \kg$ if and only if $\Vert \kf-\kg \Vert=0$.
Then $L^1_\om(G,X)$ is a Banach space with the above norm; it is called the Banach space of Bochner $\om\lambda$-integrable functions from $G$ into $X$ (see \cite[Appendix B.12]{DF} or \cite[Section II.2]{DU} for the details). It follows from the definition of Bochner integrable functions that the vector space
of the equivalent classes of integrable simple functions from $G$ into $X$, i.e. all the elements $\kf \from G \to X$ of the form
 $$\kf(\cdot) \sim \dss\sum_{i=1}^n \chi_{A_i}(\cdot) x_i,  \ \ \ A_i \ \text{is}  \ \lambda-\text{integrable}, \ \ \ x_i\in X $$
is dense in $L^1_\om(G,X)$ \cite[Appendix B.12]{DF} or \cite[Section II.2]{DU}. 

Let $C_{00}(G,X)$ denotes the set of all continuous functions from $G$ into $X$ with compact support. Then, for each $f\in C_{00}(G,X)$, $\im f$ is a compact metric space, and so, it is separable. Hence $f$ is $\lambda$-measurable from \cite[Appendix B.11(c)]{DF} or \cite[Theorem II.1.2]{DU}. Moreover, $C_{00}(G,X)$ approximates $X$-valued simple functions in norm, and so, it is norm-dense in $L^1_\om(G,X)$ (see \cite[Theorem B.11(d)]{DF} for the details).

The following proposition shows that, in certain cases, there is a convolution multiplication or action on $L^1_\om(G,X)$.

\begin{prop}\label{P: Bochner-Vector Alg}
Let $G$ be a locally compact group, let $\om$ be a weight on $G$, let $A$ be a Banach algebra, and let $X$ be a Banach left $A$-module. Then:\\
{\rm (i)} $L^1_\om(G,A)$ becomes a Banach algebra with the convolution multiplication $*$ specified by
$$ (\f*\g )(s)=\int_G \f(t)\g(s-t)dt \ \ \ (\f,\g \in C_{00}(G,A));$$
{\rm (ii)} $L^1_\om(G,X)$ becomes a Banach left $L^1_\om(G,A)$-module with the action specified by
$$ (\f*\g)(s)=\int_G \f(t)\g(s-t)dt \ \ \ (\f\in C_{00}(G,A), \g\in C_{00}(G,X)).$$
\end{prop}

\begin{proof}
(i) Let $\f,\g\in C_{00}(G,A)$. It is clear that, for every $s\in G$, the map
$t\mapsto \f(t)\g(s-t)$ belongs to $C_{00}(G,A)$. Hence it is $\lambda$-measurable.
Thus, by \cite[Appendix B.13]{DF}, the Bochner integral $\int_G \f(t)\g(\cdot-t)dt$ exists. Moreover, similar to \cite[Section 3.5]{RS}, it can be shown that the map $s\mapsto (\f*\g )(s)$ is continuous and

\begin{eqnarray*}
 \int_G \Vert (\f*\g)(s)\Vert_A \om(s) ds &\leqslant & \int_G \int_G \Vert \f(t)\g(s-t)\Vert_A \om(s) dtds
\\ &\leqslant & \int_G \int_G \Vert \f(t) \Vert_A \Vert \g(s-t)\Vert_A \om(t)\om(s-t) dsdt
\\ & = & \Vert \f\Vert \Vert \g \Vert.
\end{eqnarray*}
Therefore $\f*\g\in L^1_\om(G,A)$ and, by \cite[Appendix B.13]{DF}, $\Vert \f*\g\Vert \leqslant \Vert \f\Vert \Vert \g\Vert$. The final results follows from that of continuity and the fact that $C_{00}(G,A)$ is dense in  $L^1_\om(G,A)$. 

The proof of (ii) is similar to (i).
\end{proof}

Let $(X, ||\cdot||_X)$ be a Banach space, let $\omega$ be a weight on $G$, and let $L^1_\om(G)\widehat{\otimes} X$ be the projective tensor product of $L^1_\om(G)$ and $X$. Let $f\in C_{00}(G)$ and $x\in X$, and consider the function $\f_x \from G \to X$ defined by
$$ \f _x (t)=f(t)x  \ \  (t\in G).$$
Clearly $ \f_x\in C_{00}(G,X)$ and $\Vert  \f_x\Vert = \Vert f  \Vert  \Vert  x \Vert $. Hence the map $\alpha_X \from C_{00}(G)\times X \to C_{00}(G,X)$ definded by
$$ \alpha_X(f\otimes x)=\f_x \ \ \ (f\in C_{00}(G), x\in X),$$
is well-definded, bilinear, and $\Vert \alpha_X(f\otimes x) \Vert= \Vert f  \Vert  \Vert  x \Vert $.
Therefore, by continuity, there is a unique operator, denoted also by  $\alpha_X$, from $L^1_\om(G)\widehat{\otimes} X$ into $L^1_\om(G,X)$ such that  $$ \alpha_X(f\otimes x)(t)=f(t)x \ \ \ (f\in C_{00}(G), t\in G, x\in X).$$
It is shown in \cite[Proposition 3.3, page 29]{DF} that $\alpha_X$ is a isometric linear isomorphism.

Now suppose that $A$ is a Banach algebra and that $X$ is a Banach left $A$-module. It is well-known that $L^1_\om(G)\widehat{\otimes} A$ turns into a Banach algebra along with the action
specified by
$$ (f\otimes a)(g\otimes b)=f*g\otimes ab \ \ \ (f,g\in L^1_\om(G), a,b \in A),$$
and $L^1_\om(G)\widehat{\otimes} X$ turns into a Banach left $L^1_\om(G)\widehat{\otimes} A$ -module with the action specified by
$$ (f\otimes a)(g\otimes x)=f*g\otimes ax \ \ \ (f,g\in L^1_\om(G), a \in A, x\in X).$$

The following theorem shows that, through $\alpha_X$, the preceding actions
coincide with their corresponding vector-valued convolution ones.

\begin{thm}\label{T: Alpha-algebraic}
Let $G$ be a locally compact group, let $\omega$ be a weight on $G$, let $A$ be a Banach algebra, and let $X$ be a Banach $A$-module. Then:\\
{\rm (i)}  for every $u\in L^1_\om(G)\widehat{\otimes} A$ and $v\in L^1_\om(G)\widehat{\otimes} X$,
$$  \alpha_X (uv)=\alpha_A(u)* \alpha_X(v);$$\\
{\rm (ii)}  $\alpha_A$ is an isometric algebraic isomorphism from $L^1_\om(G)\widehat{\otimes} A$ onto $L^1_\om(G,A)$.
\end{thm}

\begin{proof}
(i) It suffices to show that, for every $f,g \in
C_{00}(G)$, $a\in A$ and $x\in X$,
$$  \alpha_X (f*g\otimes ax)=\alpha_A(f\otimes a)* \alpha_X(g\otimes x).$$
Let $s\in G$. Then

\begin{eqnarray*}
 \alpha_X (f*g\otimes ax)(s) &=& (f*g)(s)ax \\
&=& ax\int_G f(t)g(s-t)dt  \\
&=& \int_G [f(t)a][g(s-t)x] dt \\
&=& \int_G \alpha_A(f\otimes a)(t) \alpha_X(g\otimes x)(s-t) dt
\\ & = & [\alpha_A(f\otimes a)* \alpha_X(g\otimes x)](s).
\end{eqnarray*}
This completes the proof.\\
(ii) Follows from (i) by replacing $X$ with $A$.
\end{proof}

Let $X$ be a Banach left $L^1_\om(G)$-module, and let $\om\geqslant 1$. Let $\pi^X_\om \from L^1_\om(G)\widehat{\otimes} X \to X$ and $\phi ^X_\om \from L^1_\om(G)\widehat{\otimes} X \to X$ be the normed-decreasing operators specified by
$$ \pi^X_\om (f\otimes x)=fx, \ \ \ \phi^X_\om (f\otimes x)=[\int_G f(t)dt]x \ \ \ (f\in L^1_\om(G), x\in X).$$
When there is no risk of ambiguity, we write $\pi$ instead of $\pi^X_\om$ and $\phi$ instead of $\phi ^X_\om$.

There is a vector-valued analogous of the above maps. However we need some introduction before defining them.
  
We recall that, if $A$ is a Banach algebra, then a Banach left $A$-module $X$ is
\emph{essential} if it is the closure of $AX=\text{span} \{ax \mid
a\in A, x\in X \}$.
Suppose that $A$ has a bounded approximate identity. Then, by Cohen's factorization theorem \cite[Corollary 2.9.26]{D}, $X=AX$, and so, $A$ has a bounded left approximate identity for $X$.

Let $X$ be an essential Banach left $L^1_\om(G)$-module. Then the action of
$L^1_\om(G)$ on $X$ can be extended to $M_\om(G)$  \cite[Theorem
7.14]{DL} so that for every $x\in X$, the mapping $t\mapsto \delta_tx$ from $G$ into $X$ is continuous. The following lemma allows us to construct the vector-valued version of $\pi^X_\om$.

\begin{lem}
Let $X$ be an essential Banach left $L^1_\om(G)$-module. Then, for every $\f\in L^1_\om(G,X)$, the mapping $t\mapsto \delta_t[\f(t)]$ from $G$ into $X$ is  Bochner $\lambda$-integrable.
\end{lem}

\begin{proof}
Since there is $M>0$ such that $||\delta_t[\f(t)]||\leq M\om(t)||\f(t)||$ $(t\in G)$ and $\f\in L^1_\om(G,X)$, it suffices to show that the mapping $t\mapsto \delta_t[\f(t)]$ from $G$ into $X$ is  $\lambda$-measuarble. By \cite[Definition II.2.1]{DU}, there is a sequence of integrable simple functions $\{\f_n\}$ from $G$ into $X$ such that $$\lim_{n\to \infty} \int_G||\f_n(t)-\f(t)||d\lambda=0.$$ Without the loss of generality (by going to subsequences), this implies that $\lim_{n\to \infty} ||\f_n(t)-\f(t)||=0$ $\lambda$-almost everywhere. Therefore $$\lim_{n\to \infty} ||\delta_t[\f_n(t)]-\delta_t[\f(t)]||\leq M\om(t)\lim_{n\to \infty} ||\f_n(t)-\f(t)||=0$$
$\lambda$-almost everywhere. Hence we have the result if we show that for every $x\in X$ and a $\lambda$-measurable set $E$ with $\lambda(E)<\infty$, the mapping $$t\mapsto \delta_t[\chi_E(t)x]=\chi_E(t)\delta_tx$$ from $G$ into $X$ is  $\lambda$-measuarble. Since $\lambda(E)$ is finite, there is a sequence of compact sets $\{K_n\}$ such that $K_n\subset E$ for all $n\in \N$ and $\lambda(E\setminus \bigcup_{n=1}^\infty K_n)=0$. On the other hand, $X$ is essential so that the mapping $t\mapsto \delta_tx$ is continuous. Hence each $\{\delta_tx \mid t\in K_n\}$ is a compact metric space, and so, it is seperable. Thus $$ \{\delta_tx \mid t\in \bigcup_{n=1}^\infty K_n\}$$ is separable. Hence, by
\cite[Theorem II.1.2]{DU}, the mapping $t\mapsto \delta_t[\chi_E(t)x]$ is  $\lambda$-measuarble.
\end{proof}

Let $X$ be an essential Banach left $L^1_\om(G)$-module, and let $\Pi^X_\om \from L^1_\om(G,X) \to X$ and $\Phi^X_\om \from L^1_\om(G,X) \to X$ be the norm-decreasing operators defined by
$$ \Pi^X_\om (\f )=\int_G \delta_t[\f(t)] dt, \ \  \ \    \Phi^X_\om (\f)=\int_G \f(t) dt $$ for every $\f\in L^1_\om(G,X)$.
When there is no risk of ambiguity, we write $\Pi$ instead of $\Pi^X_\om$ and $\Phi$ instead of $\Phi ^X_\om$.

The following proposition establishes the relationship between the above maps.

\begin{prop}\label{P: pi_phi}
Let $G$ be a locally compact group, let $\omega, \sg \geq 1$ be weights on $G$, and let $X$ be an essential Banach left $L^1_\sg(G)$-module. Then:\\
{\rm (i)}  $\pi =  \Pi \circ\alpha_X$;\\
{\rm (ii)}  $\phi =\Phi  \circ \alpha_X$;\\
{\rm (iii)}  for every $\f\in L^1_\om(G,L^1_\sg(G))$ and $\g\in L^1_\om(G,X)$,
$$\Pi ^X_\om(\f*\g)=\Pi ^{L^1_\sg(G)}_\om(\f)\Pi ^X_\om(\g) \ \ \ \ \ \text{and} \ \ \ \  \Phi ^X_\om(\f*\g)=\Phi ^{L^1_\sg(G)}_\om(\f)\Phi ^X_\om(\g).$$
\end{prop}

\begin{proof}
(i) Since $X$ is essential, the action of
$L^1_\sg(G)$ on $X$ can be extended to $M_\sg(G)$ and $X$ becomes a unital Banach $M_\sg(G)$-module  \cite[Theorem
7.14]{DL}. Now to prove our result, it suffices to show that, for every $f \in C_{00}(G)$ and $x\in X$,
$$ \pi(f\otimes x)=\Pi (\alpha_X(f\otimes x)).$$
We have
\begin{eqnarray*}
\pi(f\otimes x) &=& \int_G f(t)\delta_tx dt  \\
&=& \int_G \delta_t[f(t)x]dt \\
&=& \int_G \delta_t [\alpha_X(f\otimes x)(t)]dt \\
& = & \Pi (\alpha_X(f\otimes x)).
\end{eqnarray*}
(ii) Let $f \in C_{00}(G)$ and $x\in X$. Then
$$ \phi(f\otimes x)=[\int_G f(t)dt]x=\int_G \alpha_X(f\otimes x)(t)dt=\Phi (\alpha_X(f\otimes x)).$$
The final result follows from the continuity.\\
(iii) It is straightforward to check that, for every  $u\in L^1_\om(G)\widehat{\otimes} L^1_\sg(G)$ and $v\in L^1_\om(G)\widehat{\otimes} X$,
$$\pi ^X_\om(uv)=\pi ^{L^1_\sg(G)}_\om(u)\pi ^X_\om(v) \ \ \ \ \ \text{and} \ \ \ \ \phi ^X_\om(uv)=\phi ^{L^1_\sg(G)}_\om(u)\phi ^X_\om(v).$$
Thus the result follows from parts (i), (ii) and Theorem \ref{T: Alpha-algebraic}(i).
\end{proof}

The following corollary is an immediate consequence of Proposition \ref{P: pi_phi}.

\begin{cor}\label{C: classification of kernel}
Let $G$ be a locally compact group, let $\omega \geqslant 1$ be a weight on $G$, and let $X$ be an essential Banach left $L^1_\om(G)$-module. Then:\\
$($i$)$ $\alpha_X(\ker \pi)=\ker \Pi$;\\
$($ii$)$ $\alpha_X(\ker \phi)=\ker \Phi.$
\end{cor}

\section{Translation of operators}\label{S:Trans-op}
Throughout the rest of this paper, we let $G$ be a locally compact abelian group and all modules be symmetric.
Let $\om$ be a weight on $G$. We shall consider the following auxiliary weight on $G$:
$$ \oo (t)=\om(-t) \ \ \ (t\in G). $$
In this section, we show that how can we transfer information from $\ker \Phi ^X_{\om\os}$
into $\ker \Pi ^X_\om$.

\begin{thm}\label{T: Transform TH}
Let $\omega$ and $\sg$ be weights on $G$, and let $X$ be an essential Banach $L^1_\sg(G)$-module. Then
there is an operator $\Lambda_X \from L^1_{\om\os}(G,X) \to  L^1_{\om}(G,X)$ such that:\\
$($i$)$ for every $\f \in C_{00}(G,X)$, $\Lambda_X(\f)(t)=\delta_{-t}[\f(t)]$;\\
$($ii$)$ $\Lambda_X$ is norm-decreasing linear map with the dense range;\\
$($iii$)$  if $\sg=1$, then $\Lambda_X$ is an isometric isomorphism on $L^1_{\om}(G,X)$;\\
$($iv$)$ for every $\f \in L^1_{\om\os}(G,L^1_\sg(G))$ and $g\in L^1_{\om\os}(G,X)$,
$$ \Lambda_X(\f*\g)=\Lambda_{L^1_\sg(G)}(\f)*\Lambda_X(\g).$$
\end{thm}

\begin{proof}
By \cite[Theorem 7.44]{DL}, we can assume that $\sg\geq 1$. Let $\f \in C_{00}(G,X)$ and define $\Lambda_X(\f)(t):=\delta_{-t}[\f(t)]$ for every $t\in G$. Clearly
$\Lambda_X(\f) \in C_{00}(G,X)$. Let $\Vert \cdot \Vert_{\om\os}$ and $\Vert \cdot \Vert_\om$ denote
the norms on $L^1_{\om\os}(G,X)$ and $L^1_{\om}(G,X)$, respectively. Then
\begin{eqnarray*}
\Vert\Lambda_X(\f) \Vert_\om &=& \int_G \Vert\Lambda_X(\f)(t) \Vert \om(t) dt \\
&=& \int_G\Vert \delta_{-t}[\f(t)] \Vert \om(t) dt \\
&\leqslant & \int_G \Vert \f(t) \Vert \Vert \delta_{-t}\Vert \om(t) dt  \\
&=& \int_G \Vert \f(t) \Vert \sg(-t) \om(t) dt  \\
& = & \Vert \f \Vert_{\om\os}.
\end{eqnarray*}
Hence $\Lambda_X$ can be extended to a norm-decreasing linear operator from $L^1_{\om\os}(G,X)$ into  $L^1_{\om}(G,X)$.
Moreover, it is easy to see that $\Lambda_X(C_{00}(G,X))=C_{00}(G,X)$. Hence the image of  $\Lambda_X$ is dense in $L^1_{\om}(G,X)$
i.e. (i) and (ii) hold. 

For (iv), let $\f \in C_{00}(G,L^1_\sg(G))$, $\g\in C_{00}(G,X)$, and $s\in G$. Then
\begin{eqnarray*}
[\Lambda_{L^1_\sg(G)}(\f)*\Lambda_X(\g)](s) &=& \int_G \Lambda_{L^1_\sg(G)}(\f)(t)\Lambda_X(\g)(s-t) dt \\
&=& \int_G [\delta_{-t}*\f(t)][ \delta_{t-s}\g(s-t)] dt \\
&=& \int_G [\delta_{-t}*\f(t)*\delta_{t-s}]\g(s-t) dt \\
&=& \int_G [\delta_{-s}*\f(t)]\g(s-t)  dt  \\
&=& \int_G \delta_{-s}[\f(t)\g(s-t)]  dt  \\
&=& \delta_{-s} \int_G \f(t)\g(s-t)  dt  \\
& = & \delta_{-s} [(\f*\g)(s)] \\
& = & \Lambda_X(\f*\g)(s).
\end{eqnarray*}
The final result follows from continuity. 

Finally, (iii) follows since $\Vert\Lambda_X(\f) \Vert_\om=\Vert \f \Vert_\om$ if $\sg=1$.
\end{proof}

\begin{thm}\label{T: Kernel-Agumentation}
Let $\omega$ and $\sg$ be weights on $G$ such that $\om \geqslant 1$, $\om \geqslant \sg$, $\om\os \geqslant 1$, and let $X$ be an essential
Banach $L^1_\sg(G)$-module. Then $\Lambda_X(\ker \Phi ^X_{\om\os})$ is dense in $\ker \Pi^X_\om$.
If, in addition, $\sg=1$, then $\Lambda_X(\ker \Phi ^X_\om)=\ker \Pi^X_\om$.
\end{thm}

\begin{proof}
For simplicity, we set $\Pi:=\Pi^X_\om$ and $\Phi:=\Phi ^X_{\om\os}$. We first note that, by \cite[Theorem 7.44]{DL},
we can assume that $\sg\geq 1$. Thus $X$ becomes a unital Banach $M_\sg(G)$-module \cite[Theorem
7.14]{DL} or \cite{D}. Also, from $\sg \leqslant \om$, it follows easily that $X$ is both a unital
Banach $M_\om(G)$-module and an essential Banach $L^1_\om(G)$-module.
Now let $\f \in \ker \Phi$. Then
\begin{eqnarray*}
\Pi [(\Lambda_X)(\f)] &=& \int_G \delta_t[\Lambda_X(\f)(t)] dt \\
&=& \int_G \delta_t [\delta_{-t}\f(t)]dt \\
&=& \int_G [\delta_t*\delta_{-t}]\f(t)dt \\
&=& \int_G \delta_e\f(t)dt \\
&=& \int_G \f(t)dt \\
&=& 0.
\end{eqnarray*}
Hence $\Lambda_X(\ker \Phi) \subseteq \ker \Pi$. 

Now suppose that $\f \in \ker \Pi$ and $\epsilon >0$.
There is $\g \in C_{00}(G,X)$ such that $\Vert \f - \g \Vert  < \epsilon$.
Let $U$ be a compact neighborhood of $e$, and let $\{ e_i \}_{i\in I}$ be a bounded approximate identity in $L^1_\om(G)$
for $X$.
It is well-known that $\{ e_i \}_{i\in I}$ can be chosen such that
$$ \{ e_i \}_{i\in I} \subset C_{00}(G) \ \ \ \ \text{and} \ \ \ \ \supp\ e_i \subseteq U \ \ \ (i\in I).  \eqno{(1)}$$
Let $M>0$ be an upper bound for $\{ e_i \}_{i\in I}$. Define, by induction on $n$, $x_n\in X$ and $e_n\in \{ e_i \}_{i\in I}$
such that
$$ x_0:=\Pi (\g), \ \ \Vert x_n-e_nx_n \Vert < 2^{-n}\epsilon \  \ \ \ \text{and} \ \ \ \ \ x_{n+1}:=x_n-e_nx_n . \eqno{(2)}$$
We first observe that, for each $n\in \N$, $\Vert x_n \Vert < 2^{-n+1}\epsilon$. Also
$$ \Vert x_0 \Vert = \Vert \Pi (\g) \Vert= \Vert \Pi (\f-\g) \Vert \leqslant \Vert \f-\g \Vert< \epsilon .$$
Hence
$$ \dss\sum_{n=0}^\infty \Vert e_n \Vert \Vert x_n \Vert \leqslant M\epsilon +
M \dss\sum_{n=1}^\infty 2^{-n+1}\epsilon= 3M\epsilon.$$
Thus, if we put $\h=\alpha_X\left\{\dss\sum_{n=0}^\infty  e_n \otimes x_n  \right \} $, then, from (1),
$$ \h\in C_{00}(G,X), \ \ \ \ \text{and} \ \ \ \ \Vert \h \Vert < 3M\epsilon. \eqno{(3)}$$
Moreover, from Proposition \ref{P: pi_phi}(i) and (2),
\begin{eqnarray*}
\Pi (\h)
&=& \Pi [\alpha_X(\dss\sum_{n=0}^\infty  e_n \otimes x_n)] \\
&=& \dss\sum_{n=0}^\infty e_nx_n \\
&=&  \dss\sum_{n=0}^\infty (x_n-x_{n+1}) \\
&=& x_0 \\
&=& \Pi (\g).
\end{eqnarray*}
Hence $\g-\h \in \ker \Pi  \cap C_{00}(G,X) $. Thus if we put ${\bf m}(t)=\delta_t(\g-\h)(t)$, then it follows that
${\bf m} \in \ker \Phi$ and $\Lambda_X({\bf m})=\g-\h$. Therefore
$$ \g-\h \in \Lambda_X(\ker \Phi).   $$
Finally, from (3), we have
\begin{eqnarray*}
\Vert \f - (\g-\h) \Vert
&\leqslant & \Vert \f -\g \Vert + \Vert \h \Vert \\
&\leqslant &  \epsilon + 3M\epsilon  \\
&=& (1+ 3M)\epsilon  .
\end{eqnarray*}
Hence $\Lambda_X(\ker \Phi)$ is dense in $\ker \Pi$. 

When $\sg=1$, by Theorem \ref{T: Transform TH}(iii), $\Lambda_X$ is an isometry on $L^1_\om(G,X)$. Thus $\Lambda_X(\ker \Phi ^X_\om)$ is norm-closed. Therefore it is the same as $\ker \Pi ^X_\om$.
\end{proof}

\section{the space of derivations and the augmentation ideal}\label{S:der space-aug ideal}

Let $A$ be a Banach algebra, and let $\varphi$ be a character on the Banach algebra $A$ (i.e.
$\varphi$ is a non-zero multiplicative linear functional on $A$). Then $\C$ is a Banach $A$-module
for the product defined by
$$ a\cdot z= z \cdot a= \varphi (a)z  \ \ \ (a\in A, z\in \C);$$
this one-dimensional module is denoted by $\C_\varphi$. A derivation from $A$ into $\C_\varphi$ is called a {\it point derivation at $\varphi$}; it is a linear functional $d$ on $A$ such that
$$d(ab)=d(a)\varphi(b)+\varphi(a)d(b) \ \ (a,b\in A).$$

It is well-known that if $A$ is weakly amenable, then there is no non-zero continuous point derivation on $A$ \cite[Theorem 2.8.63(ii)]{D}, i.e. $\mathcal{Z}^1(A, \C_\varphi)=\{0\}$ for every non-zero multiplicative linear functional $\varphi$ on $A$. However the converse is not true! For example, if $\mathfrak{S}$ is the group of rotation of $\mathbb{R}^3$, then, by \cite[Corollary 7.3]{J3}, the Fourier algebra $A(\mathfrak{S})$ is not weakly amenable but we know that it has no non-zero continuous point derivation  \cite[Proposition 1]{F}. Other examples include suitable
lipschitz algebra and Beurling algebras \cite{BCD}.

Our purpose in this section is to show that for Beurling algebras a weaker version of the converse is true. We start by considering the following well-known character on $M_\om(G)$, and the point derivations on it.

\begin{defn}
Let $\omega$ be a weight on $G$ such that $\om \geqslant 1$. Then the map
$\varphi^\om_0 \from M_\om(G) \to \C$ defined by $$\varphi^\om_0(\mu)=\mu(G) ,$$
is {\it the augmentation character on} $M_\om(G)$ and its kernel in $L^1_\om(G)$ is the {\it augmentation ideal of} $L^1_\om(G)$.
\end{defn}

We will show that not having continuous non-zero point derivations on the
augmentation character will determine derivation spaces for a large
cases of modules. The following lemma indicates the relationship
between the augmentation ideal and the kernel of $\phi^X_\om$.

\begin{lem}\label{L:augmentation-Kernel-module}
Let $\omega \geqslant 1$ and $\sg$ be weights on $G$, let $X$ be a Banach $L^1_\sg(G)$-module, and let $I_0$ be the augmentation ideal of $L^1_\om(G)$. Let $\iota\otimes id_X \from I_0 \widehat{\otimes}  X \to  L^1_{\om}(G) \widehat{\otimes} X$ be the norm-decreasing linear operator specified by
$$ \iota\otimes id_X(f\otimes x)=f\otimes x \ \ \ (f\in I_0, x\in X).$$
Then $\iota\otimes id_X$ is a bi-continuous algebraic isomorphism from $I_0 \widehat{\otimes}  X$ onto
$\ker \phi^X_\om$.
\end{lem}

\begin{proof}
It is clear that $\iota\otimes id_X( I_0 \widehat{\otimes}  X) \subseteq \ker \phi ^X_{\om}$. On the other hand,
let $u=\dss\sum_{n=1}^\infty  f_n \otimes x_n \in \ker \phi ^X_{\om}$. Fix $g\in C_{00}(G)$ with $\int_G g(t)dt=1$, and consider the continuious projection $P \from L^1_\om(G) \to I_0$ defined by
$$ P(f)=f-g \int_G f dt \ \ (f\in L^1_\om(G)).$$
Since $\phi ^X_{\om}(u)=0$, we have
\begin{eqnarray*}
u
&=& \dss\sum_{n=1}^\infty  P(f_n) \otimes x_n + \dss\sum_{n=1}^\infty  (g\int_G f_n(t)dt \otimes x_n)\\
&=&  \dss\sum_{n=1}^\infty  P(f_n) \otimes x_n+ \dss\sum_{n=1}^\infty  (g\otimes [\int_G f_n(t)dt] x_n)\\
&=&  \dss\sum_{n=1}^\infty  P(f_n) \otimes x_n+ g  \otimes \dss\sum_{n=1}^\infty [\int_G f_n(t)dt] x_n\\
&=&  \dss\sum_{n=1}^\infty   P(f_n) \otimes x_n+ g  \otimes  \phi ^X_{\om}(u) \\
&=& \dss\sum_{n=1}^\infty    P(f_n) \otimes x_n.
\end{eqnarray*}
Hence $u\in \iota\otimes id_X(I_0 \widehat{\otimes}  X)$. Thus $\iota\otimes id_X( I_0 \widehat{\otimes}  X)=\ker \phi ^X_{\om}$. Moreover, $P\otimes id_X$ is the right inverse of $\iota\otimes id_X$, i.e. $\iota\otimes id_X$ is one-to-one. The bi-continity follows from the open mapping theorem.
\end{proof}

\begin{thm}\label{T: Kernel-module}
Let $\omega$ and $\sg$ be weights on $G$ such that  $\om \geqslant 1$, $\om \geqslant \sg$, $\om\os \geqslant 1$, and let $X$ be an essential Banach $L^1_\sg(G)$-module.  Suppose that $\mathcal{Z}^1(L^1_{\om\os}(G), \C_{\varphi^{\om\os}_0})=\{0\}$. Then:\\
{\rm (i)}  $\ker \Phi^X_{\om\os}$ is an essential Banach $\ker \Phi^{L^1_\sg(G)}_{\om\os}$-module;\\
{\rm (ii)}   $\ker \Pi^X_\om$ is an essential Banach $\ker \Pi^{L^1_\sg(G)}_\om$-module.

\end{thm}

\begin{proof}
It follows from Proposition \ref{P: pi_phi}(iii) that $\ker \Phi^X_{\om\os}$ is a Banach $\ker \Phi^{L^1_\sg(G)}_{\om\os}$-module. Thus it remains to show the essentiality. 

Let $I_0$ be the augmentation ideal of $L^1_{\om\os}(G)$.
Since $L^1_{\om\os}(G)$ has no non-zero continuous point derivations at $\varphi^{\om\os}_0$, $I_0^2:=I_0I_0$ is dense in $I_0$ \cite[Proposition 1.8.8]{D}. Hence, from the essentiality of $X$,
$$I_0\widehat{\otimes} X=[(I_0 \widehat{\otimes} L^1_{\sg}(G)) ( I_0 \widehat{\otimes}  X )]^-.  $$
Therefore, by Lemma \ref{L:augmentation-Kernel-module}, if we apply
$\iota\otimes id_X$ to the both sides of the above equality,
we have
$$ \ker \phi ^X_{\om\os}=[\ker \phi ^{L^1_{\sg}(G)}_{\om\os} \ker \phi ^X_{\om\os}]^-. $$
The final result follows from applying $\alpha_X$ to the both sides of the preceding equality and using
Theorem \ref{T: Alpha-algebraic}(i) and Corollary \ref{C: classification of kernel}(ii).\\
(ii) It follows from part (i), Theorem \ref{T: Transform TH}(iv), and Theorem \ref{T: Kernel-Agumentation}.
\end{proof}

Let $A$ be a Banach algebra, and $X$ be a Banach $A$-bimodule. There is a
Banach $A$-bimodule actions on $A \widehat{\otimes} X$ specified by
$$ a \cdot(b \otimes x)=ab \otimes x \ \ , \ \ (b \otimes x)\cdot a=b \otimes
xa \ \ (a,b \in A, x\in X).$$

\begin{cor}\label{C: Characterization-kernel}
Let $\omega$ and $\sg$ be weights on $G$ such that  $\om \geqslant 1$, $\om \geqslant \sg$, $\om\os \geqslant 1$, and let $X$ be an essential Banach $L^1_\sg(G)$-module. Suppose that $\mathcal{Z}^1(L^1_{\om\os}(G), \C_{\varphi^{\om\os}_0})=\{0\}$. Then
$$\ker \pi^X_\om=\overline{\spn} \{ f\cdot u-u\cdot f \mid f \in L^1_\om(G), u\in \ker \pi^X_\om \}.$$
\end{cor}

\begin{proof}
From Theorem \ref{T: Kernel-module}(ii),
Theoem \ref{T: Alpha-algebraic}(i), and Proposition \ref{P: pi_phi}(i), we have
$$ \ker \pi ^X_\om=[\ker \pi ^{L^1_\sg(G)}_\om \ker \pi ^X_\om]^-. $$
Hence it suffices to show that
$$ [\ker \pi ^{L^1_\sg(G)}_\om \ker \pi ^X_\om]^- =\overline{\spn} \{ f\cdot u-u\cdot f \mid f \in L^1_\om(G),
u\in \ker \pi^X_\om \}. \eqno{(1)}$$
Let $\{ e_i \}_{i\in I}$ be a bounded approximate identity in $L^1_\om(G)$ for $L^1_\om(G)$, $L^1_\sg(G)$, and $X$. For every $f \in L^1_\om(G)$ and $u\in \ker \pi^X_\om$,
$$ f\cdot u-u\cdot f=\dss\lim_{i\to \infty} (f\otimes e_i-e_i\otimes f)u \in [\ker \pi ^{L^1_\sg(G)}_\om \ker \pi ^X_\om]^-.$$
Hence ``$\supseteq$" follows in (1).

Conversly, let  $m\in \ker \pi ^X_\om$ and $v=\dss\sum_{n=1}^\infty  f_n \otimes g_n \in \ker \pi ^{L^1_\sg(G)}_\om$. We have
\begin{eqnarray*}
(e_i \otimes e_i) v
&=& \dss\sum_{n=1}^\infty  e_i f_n \otimes e_i g_n  \\
&=&  \dss\sum_{n=1}^\infty [(e_i \otimes g_n) (f_n \otimes e_i-e_i \otimes f_n +  e_i \otimes f_n )]\\
&=&  \dss\sum_{n=1}^\infty [(e_i \otimes g_n) (f_n \otimes e_i-e_i \otimes f_n)] +  e_i^2 \otimes \dss\sum_{n=1}^\infty  f_n g_n\\
&=&   \dss\sum_{n=1}^\infty [(e_i \otimes g_n) (f_n \otimes e_i-e_i \otimes f_n)] + e_i \otimes \pi ^{L^1_\sg(G)}_\om(v) \\
&=&   \dss\sum_{n=1}^\infty [(e_i \otimes g_n) (f_n \otimes e_i-e_i \otimes f_n)].
\end{eqnarray*}
However, it is straightforward to check that, for each $n\in \N$, $m\cdot g_n \in \ker \pi ^X_\om$ and
$$  [(e_i \otimes g_n) (f_n \otimes e_i-e_i \otimes f_n)]m \rightarrow  f_n\cdot [m\cdot g_n]-[m\cdot g_n] \cdot f_n $$
as $i \rightarrow \infty$. Hence
$$  [(e_i \otimes g_n) (f_n \otimes e_i-e_i \otimes f_n)]m \in \overline{\spn} \{ f\cdot u-u\cdot f \mid f \in L^1_\om(G),
u\in \ker \pi^X_\om \}.$$
The final result follows since $(e_i \otimes e_i) vm \rightarrow vm$ as $i \rightarrow \infty$.
\end{proof}

\begin{thm}\label{T:Beurling alg-der}
Let $\omega$ and $\sg$ be weights on $G$ such that  $\om \geqslant 1$, $\om \geqslant \sg$, $\om\os \geqslant 1$. Suppose that $\mathcal{Z}^1(L^1_{\om\os}(G), \C_{\varphi^{\om\os}_0})=\{0\}$. Then, for any Banach $L^1_\sg(G)$-module $\X$, every continuous derivation from $L^1_\om(G)$ into $\X$ is zero.
\end{thm}

\begin{proof}
First consider the case $\X=X^*$, where $X$ is an essential Banach  $L^1_\sg(G)$-module.

Let $D \from L^1_\om(G) \To X^*$ be a bounded derivation. Define the
bounded operator $\tilde{D} \from L^1_\om(G) \widehat{\otimes}  X \to \C$ specified
by
$$\tilde{D}(f \otimes x)=\la D(f) \ , \ x \ra  \ \ \ \ (f\in L^1_\om(G), x\in X).$$
We first claim that
$$\tilde{D}=0 \ \ \ \text{on} \ \ \ \ker \pi^X_\om.  \eqno{(1)} $$
A straightforward calculation shows that, for all
$f,g,h\in L^1_\om(G)$ and $x\in X$,
$$\D [f\cdot (gh\otimes x-g\otimes hx)-(gh\otimes x-g\otimes hx)\cdot f]=0.  $$
On the other hand, from Corollary \ref{C: Characterization-kernel},
$$\ker \pi^X_\om=\overline{\spn} \{ f\cdot u-u\cdot f \mid f \in L^1_\om(G), u\in \ker \pi^X_\om \}.$$
Hence (1) follows if we show that
$$\ker \pi^X_\om=\overline{\spn} \{ gh\otimes x-g\otimes hx \mid g,h \in L^1_\om(G), x\in X\}.  \eqno{(2)}$$

Let $u=\dss\sum_{n=1}^\infty  f_n \otimes x_n \in \ker \pi ^X_\om$. Let $\{ e_i \}_{i\in I}$ be a bounded approximate identity in $L^1_\om(G)$ for both $L^1_\om(G)$ and $X$. Then
\begin{eqnarray*}
e_i\cdot u
&=& \dss\sum_{n=1}^\infty  (e_i f_n \otimes x_n-e_i \otimes f_n x_n +e_i \otimes f_n x_n) \\
&=&  \dss\sum_{n=1}^\infty (e_i f_n \otimes x_n-e_i \otimes f_n x_n) + e_i \otimes \dss\sum_{n=1}^\infty  f_n x_n\\
&=&  \dss\sum_{n=1}^\infty (e_i f_n \otimes x_n-e_i \otimes f_n x_n) + e_i \otimes \pi ^X_\om(u) \\
&=&  \dss\sum_{n=1}^\infty (e_i f_n \otimes x_n-e_i \otimes f_n x_n).
\end{eqnarray*}
Thus (2) follows since $e_i \cdot u \rightarrow u$ as $i \rightarrow \infty$. Hence, since $D$ is a derivation, for all $g\in L^1_\om(G)$ and $x\in X$,
\begin{eqnarray*}
\la D(g) \ , \ x \ra  &=&  \D(g  \otimes x) \\
&=& \dss\lim_{i\to\infty} \D(g \otimes xe_i)  \\
&=& \dss\lim_{i\to\infty} \la  e_iD(g) \ , \ x \ra \\
&=&  \dss\lim_{i\to\infty} [ \la D(e_ig) \ , \ x \ra - \la D(e_i)g \ , \ x \ra ] \\
&=& \dss\lim_{i\to\infty} \D(e_ig \otimes x- e_i \otimes gx)\\
&=&  0,
\end{eqnarray*}
where the final equality follows from (1). Hence $D=0$.

The general case follows by adapting an argument similar to the one made in the proof of \cite[Theorem 2.8.63(iii)]{D}.
\end{proof}

\section{weak amenability}\label{S:weak amen}

In this section, we present our main results on weak amenibility. We first recall that if $\omega$ is a weight on $G$ such that $\om \geqslant 1$, then the {\it strong operator topology} on
$M_\om(G)$ is defined as follows: a net $\{\mu_{\alpha}\}$ converges
to $\mu$ ($\mu_{\alpha} \overset{s.o.}{\To} \mu$) if and only if
$\mu_{\alpha}*f \to \mu*f$ in norm
for every $f\in L^1_\om(G)$. From \cite[Lemma 13.5]{DL}, both $L^1_\om(G)$ and $l^1_\om(G)$ are s.o. dense in $M_\om(G)$. 

The following proposition gives a sufficient condition on $\om$ so that $L^1_\om(G)$ has no non-zero continuous point derivation. This has been indicated in various references for particular cases but we have not seen it in its general form, so it seems worthwhile to provide a complete proof of it.

\begin{prop}\label{P:Point Der-Beurling alg}
Let $\omega$ be a weight on $G$ such that $\om \geqslant 1$. Then\\
$($i$)$ $\mathcal{Z}^1(L^1_\om(G),\C_{\varphi^\om_0})=\{0\}$ whenever $\mathcal{Z}^1(l^1_\om(G),\C_{\varphi^\om_0})=\{0\}$.\\
$($ii$)$ Suppose that, for every $t\in G$, $\inf \{ \omega (nt)/n \mid n\in \N \}=0$. Then $\mathcal{Z}^1(L^1_\om(G),\C_{\varphi^\om_0})=\{0\}$.
\end{prop}

\begin{proof}
(i) Let $d \from L^1_\om(G) \to \C$ be a continuous point derivations at $\varphi^\om_0$. By
\cite[Theorem 2.9.53]{D} and \cite[Theorem 7.14]{DL}, there is a unique extension of $d$ to a continuous
derivation $\tilde{d} \from M_\om(G) \to \C$; the action of $M_\om(G)$ on $\C$ is defined by
$$ \mu \cdot z=z\cdot \mu=\mu(G)z \ \ \ (\mu\in M_\om(G), z\in \C).$$
Moreover, $\tilde{d}$ is continuous with respect to the strong operator topology on $M_\om(G)$.
Clearly the restriction of $\tilde{d}$ to $l^1_\om(G)$ belongs to $\mathcal{Z}^1(l^1_\om(G),\C_{\varphi^\om_0})$. Therefore $\tilde{d}=0$ on
$l^1_\om(G)$. However, $l^1_\om(G)$ is dense in $M_\om(G)$ with respect to the strong operator
topology. Hence $\tilde{d}=0$, and so $d=0$.

(ii) Take $d\in \mathcal{Z}^1(l^1_\om(G),\C_{\varphi^\om_0})$ and $t\in G$. For every $n\in \N$,
$d(\delta_{nt})=n[\delta_{(n-1)t} \cdot d(\delta_t)]=nd(\delta_t)$. Hence
$$ \Vert d(\delta_t) \Vert \leqslant  \Vert d \Vert \Vert \delta_{nt} \Vert / n =
 \Vert d \Vert \om(nt) /n.$$
Thus, from the hypothesis, $d(\delta_t)=0$, i.e. $d=0$.

The final results follows from part (i).
\end{proof}

We note that for any weight $\om$ on $G$, its {\it symmetrization} is the weight defined by
$\Omega(t):=\omega (t)\omega (-t)$ ($t\in G$). We can now use Theorem \ref{T:Beurling alg-der} to present a class of weakly amenable Beurling algebras. This has been already established, by N. Groenbaek, for Beurling algebras on  discrete abelian groups \cite{G1}.

\begin{thm}\label{T:weak amenable}
Let $\omega$ be a weight on $G$ such that, for every $t\in G$,
$\inf \{ \Omega (nt)/n \mid n\in \N \}=0$. Then
$L^1_\om(G)$ is weakly amenable.
\end{thm}

\begin{proof}
Let $\R^{+\bullet}:=(0, \infty)$ be the group of positive real numbers with respect to multiplication. By \cite[Theorem 7.44]{DL}, there is a continuous character (i.e. a non-zero group homomorphism) $\chi \from G \to \R^{+\bullet}$ such that $\om_1:=\om/\chi$ is a weight on $G$, $\om_1 \geqslant 1$, and $L^1_\om(G)$ is isometrically isomorphic to $L^1_{\om_1}(G)$. Therefore it suffices to show
that $L^1_{\om_1}(G)$ is weakly amenable. Since $\chi$ is a group homomorphism, for every $t\in G$,
$$\omega_1(t) \tilde{\omega_1}(t)=\omega (t)\oo(t)=\Omega(t).$$
On the other hand, by Proposition \ref{P:Point Der-Beurling alg}, there is no non-zero continuous point derivation on $L^1_\Omega(G)$ at $\varphi^\Omega_0$. Hence, from Theorem \ref{T:Beurling alg-der}, every continuous derivation from $L^1_{\om_1}(G)$ into any Banach
$L^1_{\om_1}(G)$-module is zero, i.e. $L^1_{\om_1}(G)$ is weakly amenable.
\end{proof}

We would like to point out that the condition in Proposition
\ref{P:Point Der-Beurling alg} is not necessary for vanishing of
$\mathcal{Z}^1(L^1_\om(G), \C_{\varphi^\om_0})$. Indeed, in Theorem
\ref{T:2-weak amenable-sharp weights}, we will present examples of
Beurling algebras with sharp growing weights which have no non-zero
continuous point derivations.

\section{2-weak amenability}\label{S:2-weak amen}

Let $\om\geq 1$ be a weight on $G$. Define the function $\om_1$ on $G$ by

\begin{eqnarray*}
\om_1(s) &=& \dss\limsup_{t\to\infty} \frac{\om(t+s)}{\om(t)}  \\
&=& \inf \left\{\sup \left \{\frac{\om(t+s)}{\om(t)} : t\notin K \right \} \mid K \ \text{is a compact subset of G} \right \}.
\end{eqnarray*}
It is clear that $\om_1$ is a sub-additive function on $G$ such that $\om_1 \leq \om$ and $\om\tilde{\om}_1 \geq 1$. However, we do not know whether $\om_1$ is continuous. Nevertheless
we have the following lemma:

\begin{lem}\label{L:om-sg euivalent}
Let $\om \geq 1$ be a weight on $G$, and let $\om_1$ be as above. Then:\\
{\rm (i)} $\om_1$ is measurable;\\
{\rm (ii)} there is a weight on $G$, denoted by $\sg_\om$, and positive real numbers $M$ and $N$ such that
$$ M\om_1(t) \leq \sg_\om(t) \leq N\om_1(t) \ \ (t\in G).$$
In particular, $\sg_\om$ is bounded if and only if $\om_1$ is bounded.
\end{lem}

\begin{proof}
(i) Let $r>0$, $G_r:=\{x\in G \mid \om_1(x)<r \}$ and $s\in G_r$. There is a compact subset $K$ of $G$ and $0<r_1<r$ such that
$$\sup \{\frac{\om(t+s)}{\om(t)} : t\notin K \}< r_1.  \eqno{(1)}$$
Let $U$ a compact neightborhood of identity in $G$ satisfying
$$\sup \{\om(s_1) : s_1\in U \}r_1<r. \eqno{(2)}$$
This is possible because $\om$ is continuous and $\om(e)=1$.
Now take $s_1\in U$. From (1) and the fact that $\om$ is sub-additive, we have
$$\sup \{\frac{\om(t+s+s_1)}{\om(t)} : t\notin K \}\leq \sup \{\frac{\om(t+s)}{\om(t)} : t\notin K \}
\om(s_1) < r_1\om(s_1).$$
Hence it follows from (2) that
$$\sup \{\frac{\om(t+s+s_1)}{\om(t)} : t\notin K, s_1\in U \}\leq \sup \{\om(s_1) : s_1\in U \}r_1<r.$$
This implies that $\om_1(s+s_1)<r$ for all $s_1\in U$ i.e. $s+U\subseteq G_r$. Thus $G_r$ is open. Hence $\om_1$ is measurable.

(ii) It follows from (i) and \cite[Definition 3.7.1 and Theorem 3.7.5]{RS}.
\end{proof}

The importance of $\sg_\om$ is presented in the following lemma in which we show that a certain Banach $L^1_\om(G)$-module can be regarded as an $L^1_{\sg_\om}(G)$-module. This interesting phonomenon helps us to connect the space of derivations into the second dual of $L^1_\om(G)$ to the behavior of the augmentation ideals of $L^1_{\om\sg_\om}(G)$. Indeed, we will see in Theorem \ref{T:2-weak amenable} that, for the case when $\sg_\om$ is bounded, this gives us a precise correspondence between continuous point derivations at $\varphi^{\om}_0$ and 2-weak amenability of $L^1_\om(G)$.

We recall from \cite[P. 77]{DL} that $L^\infty_{1/\om}(G)$, as the dual of $L^1_\om(G)$, is a Banach  $M_\om(G)$-module. In particular, for each $f\in L^\infty_{1/\om}(G)$, we have
$$f\cdot \delta_t=\delta_t \cdot f=\delta_{-t}* f \ \ \ (t\in G).$$
Moreover, $LUC_{1/\om}(G)=L^1_\om(G)L^\infty_{1/\om}(G)$.

\begin{lem}\label{L:om-sg module}
Let $\om \geq 1$ be a weight on $G$, and let $X_\om=LUC_{1/\om}(G)/C_{0,1/\om}(G)$. Then the standard action of $G$ on $X_\om$ extends continuously to an action of $M_{\sg_\om}(G)$ on $X_\om$. In particular, $X_\om$ is a unital Banach $M_{\sg_\om}(G)$-module and an essential Banach $L^1_{\sg_\om}(G)$-module.
\end{lem}

\begin{proof}
For simplicity, put $\sg={\sg_\om}$. 

By Lemma \ref{L:om-sg euivalent}, there is $M>0$ such that
$\om_1\leq M^{-1}\sg_\om$. We first show that for every $t\in G$ and $\textbf{x}\in X_\om$,
$$\left\| \delta_t\cdot \textbf{x} \right\|\leq M^{-1}\left\| \textbf{x} \right\|\sg(t).  \eqno{(1)} $$
Let $\epsilon > 0$. There is a compact set $F_t$ in $G$ such that, for every $s \notin F_t$,
$\om(s+t)/\om(s) \leq M^{-1}\sg(t)+\epsilon$. Pick a continuous function $f_t$ on $G$ with compact support such that
$$ 0 \leq f_t \leq 1 \ , \ f_t=1 \ \text{on} \ F_t. \eqno{(2)}  $$
Let $\textbf{x}=\tilde{g}$ where $g\in LUC_{1/\om}(G)$. Since $\delta_t\cdot g-(1-f_t)(\delta_t\cdot g)= f_t(\delta_t\cdot g) \in C_{0,1/\om}(G)$,
$$\left\|\delta_t\cdot \textbf{x} \right\| = \left\| \widetilde{(\delta_t\cdot g)} \right\| \leq \left\| (1-f_t)(\delta_t\cdot g) \right\|_{1/\om}.$$
On the other hand, from (2),

\begin{eqnarray*}
\left\| (1-f_t)(\delta_t\cdot g) \right\|_{1/\om} & \leq & \sup
\{ \frac{| (\delta_t\cdot g)(s)|}{\om(s)} \mid s\notin F_t \} \\
&=& \sup
\{ \frac{| g(s+t)|}{\om(s)} \mid s\notin F_t \}  \\
&=& \sup
\{ \frac{| g(s+t)|}{\om(s+t)}\frac{ \om(s+t)}{\om(s)} \mid s\notin F_t \}\\
&\leq& \left\|g \right\|_{1/\om} (M^{-1}\sg(t)+\epsilon).
\end{eqnarray*}
Therefore
$$ \left\| \delta_t\cdot \textbf{x} \right\| \leq \left\|g \right\|_{1/\om} (M^{-1}\sg(t)+\epsilon),$$
and so, (1) follows since $ \left\| \textbf{x} \right\|=\inf \{ \left\|g \right\|_{1/\om} \mid \tilde{g}=\textbf{x}
\}$ and $\epsilon$ was arbitrary.

Now suppose that $\mu \in M_\om(G)$. By \cite[Proposition 7.15]{DL}, the map
$$ s\mapsto \delta_s \cdot \textbf{x}, \ \ \ G \rightarrow X_\om $$
is continuous. Hence it is $|\mu|$-measurable. Moreover, from (1),
$$\int_G \left\| \delta_s \cdot \textbf{x} \right\| d|\mu| \leq M^{-1}\int_G \left\| \textbf{x} \right\| \sg(s) d|\mu|
=M^{-1} \left\| \textbf{x} \right\|\left\| \mu \right\|_\sg.$$
Therefore the Bohner integral $\mu \cdot \textbf{x}= \int_G \delta_s \cdot \textbf{x}  d\mu$ is well-defined and $\left\| \mu \cdot \textbf{x} \right\| \leq M^{-1}\left\| \textbf{x} \right\|\left\| \mu \right\|_\sg$ \cite[Appendix B.12]{DF}. 

The essentiality of $X_\om$ follows simply because $X_\om$ is the closure of $C_{00}(G)\cdot X_\om$.
\end{proof}

\begin{thm}\label{T:2-weak amenable-weaker version}
Let $\omega\geqslant 1$ be a weight on $G$. Suppose that $\mathcal{Z}^1(L^1_{\om\sg_\om}(G),\C_{\varphi^{\om\sg_\om}_0})=\{0\}$.
Then $L^1_\om(G)$ is $2$-weakly amenable.
\end{thm}

\begin{proof}
Let $D \from L^1_\om(G) \to L^1_\om(G)^{**}$ be a bounded derivation, and let $$X_\om=LUC_{1/\om}(G)/C_{0,1/\om}(G).$$ By the first paragragh in the proof of
\cite[Theorem 13.1]{DL}, $\im D \subseteq (C_{0,1/\om}(G))^\bot$, where
$$ C_{0,1/\om}(G)^\bot=\{ M\in L^1_\om(G)^{**} \mid M=0 \ \text{on} \ C_{0,1/\om}(G) \}.$$
Hence $D$ can be regarded as a bounded derivation from $L^1_\om(G)$ into $C_{0,1/\om}(G)^\bot=[L^\infty_{1/\om}(G)/C_{0,1/\om}(G)]^*$. 

On the other hand, since $L^1_\om(G)$ has a bounded approximate identity, by a result of Johnson \cite[Corollary 2.9.27]{D},

\begin{eqnarray*}
\mathcal{Z}^1[L^1_\om(G),(\frac{L^\infty_{1/\om}(G)}{C_{0,1/\om}(G)})^*]
&=& \mathcal{Z}^1[L^1_\om(G),(\frac{L^1_\om(G) L^\infty_{1/\om}(G)}{C_{0,1/\om}(G)})^*] \\
&=& \mathcal{Z}^1[L^1_\om(G),X_\om^*] \\
&=& 0,
\end{eqnarray*}
where the last equality follows from Lemma \ref{L:om-sg module} and Theorem \ref{T:Beurling alg-der}.
We note that, by Lemma \ref{L:om-sg euivalent}, there are $M,N>0$ such that $N^{-1}\sg_\om \leq \om$
and $\om\tilde{\sg}_\om \geq M$. Hence one can easily verify that Theorem \ref{T:Beurling alg-der} is valid for
$\om$ and $\sg_\om$.
\end{proof}

\begin{thm}\label{T:2-weak amenable}
Let $\omega\geqslant 1$ be a weight on $G$. Suppose that $\sg_\om$ is bounded. Then $\mathcal{Z}^1(L^1_\om(G),\C_{\varphi^\om_0})=\{0\}$ if and only if $L^1_\om(G)$ is $2$-weakly amenable.
\end{thm}

\begin{proof}
``$\Rightarrow$" Since $\sg_\om$ is bounded, $L^1_\om(G)$ is a dense subalgebra of $L^1_{\om\os_\om}(G)$. Hence $L^1_{\om\os_\om}(G)$ has no non-zero continuious point derivation at $\varphi^{\om\os}_0$, and so, the result follows from Theorem \ref{T:2-weak amenable-weaker version}.

``$\Leftarrow$" Let $X:=X_\om$ be as in Theorem \ref{T:2-weak amenable-weaker version}. By \cite[Proposition 2.2]{GZ}, $X^*$ is an $L^1_\om(G)$-submodule of $L^1_\om(G)^{**}$. Thus,
by hypothesis,
$$\mathcal{Z}^1[L^1_\om(G),X^*]=0. \eqno{(1)}$$
On the other hand, by Lemma \ref{L:om-sg module} and the fact that $\sg_\om$ is bounded, $X$ is an essential Banach $L^1(G)$-module. Therefore from the argument presented in the proof of Corollarly \ref{C: Characterization-kernel} (equation 1) and Theorem \ref{T:Beurling alg-der} (equation 2) we have

\begin{eqnarray*}
[\ker \pi ^{L^1(G)}_\om \ker \pi ^X_\om]^- &=& \overline{\spn} \{ f\cdot u-u\cdot f \mid f \in L^1_\om(G), u\in \ker \pi^X_\om \} \\ & \subseteq& \ker \pi^X_\om
\\ &=& \overline{\spn} \{ gh\otimes x-g\otimes hx \mid g,h \in L^1_\om(G), x\in X\}.
\end{eqnarray*}

We claim that
$$ \ker \pi ^X_\om=[\ker \pi ^{L^1(G)}_\om \ker \pi ^X_\om]^-. \eqno{(\star)}$$
Let $T\in (L^1_\om(G)\widehat{\otimes} X)^*$ such that $T=0$ on $\ker \pi ^{L^1(G)}_\om \ker \pi ^X_\om$. Hence, for every $f,g,h \in L^1_\om(G)$ and $x\in X$,
$$T[f\cdot (gh\otimes x-g\otimes hx)-(gh\otimes x-g\otimes hx)\cdot f]=0.$$
Thus if we let $\T \from L^1_\om(G) \to X^*$ be the bounded operator
defined by
$$ \la \T(f) \ , \ x \ra=T(f \otimes x) \ \ \ (f \in L^1_\om(G), x\in X),$$
then a simple calculation shows that $$\T(fgh)-f\T(gh)-\T(fg)h+f\T(g)h=0 \ \ \ (f,g,h \in L^1_\om(G)). \eqno{(2)}$$
Define the bounded operator $D \from L^1_\om(G) \to
\mathcal{B}_{L^1_\om(G)}(L^1_\om(G),X^*)$ by
$$ D(f)(g)=\T(fg)-f\T(g) \ \ (f,g\in L^1_\om(G)). \eqno{(3)}$$
From (2), it is easy to verify that $D$ is well-defined. Moreover,
upon setting
$$\la f\cdot S \ , \ x \ra =\la S\cdot f \ , \ x\ra =\la S \ , \ fx \ra,$$ the space
$\mathcal{B}_{L^1_\om(G)}(L^1_\om(G),X^*)$ becomes a Banach $A$-module and
$\mathcal{D}$ becomes a bounded derivation from $L^1_\om(G)$ into
$\mathcal{B}_{L^1_\om(G)}(L^1_\om(G),X^*)$. However, since $L^1_\om(G)$ has a bounded approximate identity, $\mathcal{B}_{L^1_\om(G)}(L^1_\om(G),X^*)$ is
isometric with $X^*$ as Banach $L^1_\om(G)$-module. Thus, from (1), $D=0$.
Therefore, from (3), $\T(fg)=f\T(g)$. So $T$ vanishes on
$$\overline{\spn} \{ gh\otimes x-g\otimes hx \mid g,h \in L^1_\om(G), x\in X\}=\ker \pi^X_\om.$$
Thus ($\star$) holds. Hence, by Theorem \ref{T: Alpha-algebraic}(i) and Corollary \ref{C: classification of kernel}(i),
$$\ker \Pi ^X_\om=[\ker \Pi ^{L^1(G)}_\om \ker \Pi ^X_\om]^-.$$
However, from Theorem \ref{T: Transform TH}(iii), $\Lambda_X$ is invertible, and so, by Theorem
\ref{T: Kernel-Agumentation},
\begin{eqnarray*}
\ker \Phi ^X_\om &=& \Lambda_X^{-1}(\ker \Pi ^X_\om) \\
&=& [\Lambda_{L^1(G)}^{-1}(\ker \Pi ^{L^1(G)}_\om)\Lambda_X^{-1}(\ker \pi ^X_\om)]^- \\
&=& [\ker \Phi ^{L^1(G)}_\om \ker \Phi ^X_\om]^-.
\end{eqnarray*}
Therefore, by Corollary \ref{C: classification of kernel}(ii),
$$ \ker \phi ^X_\om=[\ker \phi ^{L^1(G)}_\om \ker \phi ^X_\om]^-.$$
It follows from Lemma \ref{L:augmentation-Kernel-module} that $$I_0\widehat{\otimes} X=[(I_0 \widehat{\otimes} L^1(G)) ( I_0 \widehat{\otimes}  X )]^-.$$
Hence $I_0=\overline{I_0^2}$. 

This completes the proof.
\end{proof}

\begin{rem}\label{T:2-weak amenable-tensor}
Let $\{ L^1_{\om_i}(G_i) \}_{i=1}^n$ be a finite set of Beurling algebras, and let $\om:=\om_1\times \cdots \times \om_n$ be the function on $G:=G_1\times \cdots \times G_n$ defined by
$$ \om(t_1, \cdots , t_n)=\prod_{i=1}^n \om_i(t_i) \ \ (0\leq i \leq n, t_i \in G_i).$$
It is well-known that $\om$ is a weight on $G$ so that $L^1_\om(G)$ is algebraiclly isomorphic with
$\hat{\otimes}_{i=1}^n  L^1_{\om_i}(G_i)$ \cite[Proposition 1.2]{G2}. Now suppose that, for each $i$, $\om_i \geq 1$, $\sg_{\om_i}$ is bounded, and $L^1_{\om_i}(G_i)$ is 2-weakly amenable.
Then $L^1_\om(G)$ is 2-weakly amenable. Indeed, since $\sg_\om$ is bounded,
by the preceding theorem, it suffices to show that $\mathcal{Z}^1(L^1_{\om}(G),\C_{\varphi^{\om}_0})=\{0\}.$
However from our assumption and Theorem \ref{T:2-weak amenable}, $\mathcal{Z}^1(L^1_{\om_i}(G_i),\C_{\varphi^{\om}_0})$ vanishes for each $i$. Therefore it can be easily shown that $\mathcal{Z}^1(L^1_{\om}(G),\C_{\varphi^{\om}_0})=\{0\}$: this follows from the fact that
$\varphi^{\om}_0=\otimes_{i=1}^n \varphi^{\om_i}_0$.
\end{rem}

We finish this section with the following corollary, which was obtained in \cite{GZ} by a different method.

\begin{cor}\label{C:2-weak amenable-inf}
Let $\omega\geqslant 1$ be a weight on $G$. Suppose that $\sg_\om$ is bounded and $\inf \{ \omega (nt)/n \mid n\in \N \}=0$. Then $L^1_\om(G)$ is $2$-weakly amenable.
\end{cor}

\begin{proof}
By Proposition \ref{P:Point Der-Beurling alg}, there is no non-zero continuous point derivation on $L^1_\om(G)$ at $\varphi^\om_0$. Hence the result follows from the preceding theorem.
\end{proof}

\section{weights on compactly generated abelian groups}\label{S:compactly generated groups}

Let $G$ be compactly generated (abelian) group. Then, by the Structure Theorem,
$$G\cong \R^k \times \Z^m \times T \eqno{(\star)}$$ where $k,m \in \N \cup \{0\}$ and $T$ is a compact (abelian) group. Therefore we can define a continuous sub-additive function on $G$ by
$$\left| t \right|=\left\| t_1 \right\|,$$
where $t=(t_1,t_2)\in G$, $t_1 \in \R^k \times \Z^m$, and $||\cdot||$ is the Euclidean norm on $\R^{k+m}$. We can use this function to construct different weights on $G$.

For $\alpha \geq 0$ and $C>0$, let
$$\om (t)=C(1+\left| t \right|)^\alpha \ \ (t\in G). $$
It is easy to see that $\om$ is a weight on $G$; it is called {\it polynomial of degree $\alpha$}. The following theorem is a generalization of \cite[Theorem 13.2 and 13.9]{DL}.

\begin{thm}\label{T:2-weak amenable-poly}
Let $G$ be a non-compact, compactly generated group, and let $\omega$ and $\sg$ be polynomial weights of degree $\alpha$ and $\beta$, respectively. Then:\\
$($i$)$ if $\alpha\geq \beta$ and $\alpha+\beta <1$, then, for any Banach $L^1_\sg(G)$-module $\X$, every continuous derivation from $L^1_\om(G)$ into $\X$ is zero:\\
$($ii$)$ $L^1_\om(G)$ is weakly amenable if and only if $\alpha< 1/2$;\\
$($iii$)$ $L^1_\om(G)$ is $2$-weakly amenable if and only if $\alpha < 1$.
\end{thm}

\begin{proof}
(i) follows from Theorem \ref{T:Beurling alg-der} and Proposition \ref{P:Point Der-Beurling alg}(ii). 

For (ii), if $\alpha <1/2$, then, from part (i), $L^1_\om(G)$ is weakly amenable. Conversely, suppose that $\alpha \geq 1/2$. Since $G$ is not compact, it has a copy of $\R$ or $\Z$ as a direct sum. Hence there is a continuous algebraic homomorphism from $L^1_\om(G)$ onto either $L^1_{\om_{|\R}}(\R)$ or
$l^1_{\om_{|\Z}}(\Z)$. However, neither of these algebras are weakly amenable \cite[Theorem 7.43]{DL} and \cite[Corollary 5.6.19]{D}. Hence $L^1_\om(G)$ is not weakly amenable. 

Finally, for
(iii), it is easy to see that $\dss\limsup_{t\to\infty} \frac{\om(t+s)}{\om(t)}=1$. Hence, by Theorem \ref{T:2-weak amenable}, it suffices to show that there is no non-zero continuous point derivation on $L^1_{\om}(G)$ at $\varphi^{\om}_0$ if and only if $\alpha < 1$. The ``if"
part follows from Proposition \ref{P:Point Der-Beurling alg}(ii) and the ``only if" part follows from the fact that, for $\alpha \geq 1$, the Fourier transform of the elements of $L^1_{\om_{|\R}}(\R)$ and $l^1_{\om_{|\Z}}(\Z)$ are continuously differentiable \cite[Theorem 13.2 and 13.9]{DL}. This gives us a non-zero continuous point derivation on $L^1_{\om}(G)$.
\end{proof}

Another family of weights that are considered on compactly generated groups are the exponential weights. A weight $\om$ is said to be {\it exponential of degree $\alpha$}, $0\leq \alpha \leq 1$, if there exists $C>0$ such that
$$ \om(t)=e^{C|t|^\alpha}, \ \ (t\in G).$$
By our method, we can investigate the question of 2-weak amenability for these families of Beurling algebras.

\begin{thm}\label{T:2-weak amenable-exp}
Let $G$ be a non-compact, compactly generated group, and let $\om$ be an exponential weight of degree $\alpha$. Then $L^1_{\om}(G)$ is not $2$-weakly amenable if $0< \alpha < 1$.
\end{thm}

\begin{proof}
It is easy to see that, for $0< \alpha < 1$, $\dss\limsup_{t\to\infty} \frac{\om(t+s)}{\om(t)}$ is bounded by 1. Also a similar argument to that presented in Theorem \ref{T:2-weak amenable-poly}(iii) gives us a non-zero continuous point derivation on $L^1_\om(G)$. Hence the result follows from Theorem \ref{T:2-weak amenable}.
\end{proof}

\begin{rem}
(i) We note that the result of the preceeding theorem holds, with the same argument, for weights of the form $\om(t)=e^{\rho(|t|)}$, where $\rho$ is a positive increasing sub-additive function which belongs to the Lipschitz algebra on $\R^+=(0, \infty)$ with the degree $0< \alpha < 1$.\\
(ii) We would like to point out that the result of Theorem \ref{T:2-weak amenable-exp} is not true in general when
$\alpha=1$. Indeed, it is demonstrated in \cite[Theorem 13.3]{DL} that for $\om(n)=e^{|n|}$, $l^1_\om(\Z)$ is 2-weakly amenable.
\end{rem}

Theorem \ref{T:2-weak amenable-exp} can be generalized to a larger class of weights. Let $q \from \R^+ \to \R^+$ be a decreasing continuous function such that
$$\dss\lim_{r\to+\infty} q(r)=0 \ \ \text{and} \ \ \dss\lim_{r\to+\infty} rq(r)=\infty.$$
Then the function $\om \from G \to [1, \infty)$ given by
$$\om(t)= e^{|t|q(|t|)}, \ \ (t\in G)$$
is a weight on $G$. All weights constructed as above belong to a family of weights that satisfy a so-called {\it condition (S)}. This condition is defined in order to get the symmetry of the certain weighted group algebras on non-abelian groups. We refer the reader to \cite{DLC} and \cite{FGLLC} for more details.

\begin{thm}\label{T:2-weak amenable-cond S}
Let $G$ be a non-compact, compactly generated group, and let $q$ and $\omega$ be as above. Suppose that $rq(r) \geq \ln (1+r)$ for sufficiently large $r$. Then $L^1_\om(G)$ is not $2$-weakly amenable.
\end{thm}

\begin{proof}
Let $t,s \in G$ with $|t|>|s|$, and put $r=|t|-|s|$. Then $|t-s|\geq r$, and so, from the fact that $q$ is decreasing,
\begin{eqnarray*}
|t|q(|t|)-|t-s|q(|t-s|) &\leq & [r+|s|]q(r+|s|)-rq(r) \\
&=& r[q(r+|s|)-q(r)]+|s|q(r+|s|) \\
&\leq & |s|q(r+|s|) \\
&=& |s|q(|t|).
\end{eqnarray*}
Hence
$$\dss\limsup_{t\to\infty} \frac{\om(t+s)}{\om(t)}=\dss\limsup_{t\to\infty} \frac{\om(t)}{\om(t-s)} \leq \dss\limsup_{t\to\infty} e^{|s|q(|t|)}=1,$$
since $\dss\lim_{t\to\infty} q(|t|)=0$. Therefore $\sg_\om$ is bounded. On the other hand, by hypothesis, $\om(x) \geq 1+|x|$ outside a compact set. Thus the result follows in a similar way to Theorem \ref{T:2-weak amenable-exp}.
\end{proof}

Some examples of such weights are presented in \cite[Example 1.7]{DLC}. They are, for instance, given by\\

(i) $\om(t)=e^{C|t|^\alpha}=e^{|t|\frac{C}{|t|^{1-\alpha}}}, C>0, 0<\alpha<1,$

(ii) $\om(t)=e^{|t|\dss\Sigma_{n=1}^\infty \frac{c_n}{1+|t|^{\alpha_n}}}, 0<\alpha_n<1, \{ \alpha_n \}  \ \text{decreasing to 0}, \ \dss\Sigma_{n=1}^\infty c_n < \infty,$

(iii) $\om(t)=e^{C\frac{|t|}{\ln(e+|t|)}}$,

(iv) $\om(t)=e^{C\frac{|t|}{(\ln(e+|t|))^k}}, k>0$\\

In Theorem \ref{T:2-weak amenable-poly}, we gave examples of 2-weakly amenable Beurling algebras over (symmetric) polynomial weights. Now we will present a family of non-symmetric weights on $\R$ and $\Z$ for which the Beurling algebras are 2-weakly amenable and, at one side, they have a much faster growth. Let $0\leq \alpha < 1/2$ and define the functions $\om_\R$ and $\om_\Z$ on $\R$ and $\Z$, respectively, as follows:
$$\om_\R(t)=1 \ \text{if} \ t\geq 0 \ \text{and} \ \om_\R(t)=e^{|t|^\alpha} \ \text{if} \ t<0 ;$$
$$\om_\Z(n)=1 \ \text{if} \ n\geq 0 \ \text{and} \ \om_\Z(n)=e^{|n|^\alpha} \ \text{if} \ n<0 .$$

It is straightforward to verify that $\om_\R$ and $\om_\Z$ are weights.

\begin{prop}\label{P:2-weak amenable-sharp weights}
Let $\alpha$, $\om_\R$ and $\om_\Z$ be as above. Then:\\
$($i$)$ $l^1_{\om_\Z}(\Z)$ is $2$-weakly amenable;\\
$($ii$)$ $L^1_{\om_\R}(\R)$ is $2$-weakly amenable.
\end{prop}

\begin{proof}
(i) It is easy to see that, $\dss\limsup_{n\to\infty} \frac{\om_\Z(n+m)}{\om_\Z(n)}=1$. On the other hand, $l^1_{\om_\Z}(\Z)$ is a commutative regular semisimple Banach algebra on $\mathbb{T}=\{ z\in \C \mid |z|=1 \}$. It is shown in \cite{Z} that, since
$$\dss\limsup_{n\to+\infty} \ \ln\om(-n)/\sqrt{n}=0,$$
singeltons in $\mathbb{T}$ are sets of spectral synthesis for $l^1_{\om_\Z}(\Z)$. Thus there are no non-zero continuous point derivations on $l^1_{\om_\Z}(\Z)$. Hence the result follows from Theorem \ref{T:2-weak amenable}.\\
(ii) Since $\sg_{\om_\R}$ is bounded, by Theorem \ref{T:2-weak amenable} and Proposition \ref{P:Point Der-Beurling alg}(i), it suffices to show that
$$\mathcal{Z}^1(l^1_{\om_\R}(\R),\C_{\varphi^{\om_\R}_0})=\{0\}.$$
Let $d\in \mathcal{Z}^1(l^1_{\om_\R}(\R),\C_{\varphi^{\om_\R}_0})$. For every $r\in \R^+$, let $\la r \ra$ be the discrete additive subgroup of $\R$ generated by $r$.
Clearly the closed subalgebra $A_r$ of $l^1_{\om_\R}(\R)$ generated by the restriction to $\la r \ra$ is algebraically isomorphic to $l^1_{\om_\Z}(\Z)$. Thus, from (i), $d=0$ on $A_r$, and so, $d(\delta_r)=d(\delta_{-r})=0$. Hence $d=0$.
\end{proof}

The preceding theorem and Remarks \ref{T:2-weak amenable-tensor} can be routinely employed to construct fast growing weights on compactly generated groups for which the Beurling algebras are 2-weakly amenable.
For each $1\leq i \leq k$ and $1\leq j \leq m$, let $0\leq \alpha_i < 1/2$ and $0\leq \beta_j < 1/2$, and let $\om_{\alpha_i}$ and $\om_{\beta_j}$ be the weights on $\R$ and $\Z$ associated, as in Proposition \ref{P:2-weak amenable-sharp weights}, with $\alpha_i$ and $\beta_j$, respectively. Put
$$\om=\prod_{i=1}^k \om_{\alpha_i} \times \prod_{j=1}^m \om_{\beta_j}.$$ By the identification $(\star)$,
$\om$ defines a weight on $G$.

\begin{thm}\label{T:2-weak amenable-sharp weights}
Let $G$ be a compactly generated group, and let $\om$ be as above. Then $L^1_{\om}(G)$ is $2$-weakly amenable.
\end{thm}

\end{document}